\documentclass[12pt]{amsart}
\usepackage{amsmath,amssymb,amsbsy,amsfonts,latexsym,amsopn,amstext,cite,
                                               amsxtra,euscript,amscd,bm,mathabx}
\usepackage{url}
\usepackage[colorlinks,linkcolor=blue,anchorcolor=blue,citecolor=blue,backref=page]{hyperref}
\usepackage{color}
\usepackage{graphics,epsfig}
\usepackage{graphicx}
\usepackage{float}
\usepackage[english]{babel}
\usepackage{mathtools}
\usepackage{todonotes}
\usepackage{url}
\usepackage[colorlinks,linkcolor=blue,anchorcolor=blue,citecolor=blue,backref=page]{hyperref}

\usepackage[norefs,nocites]{refcheck}

\hypersetup{breaklinks=true}

\def\ccg#1{\textcolor{cyan}{#1}}

\usepackage[norefs,nocites]{refcheck}
\usepackage[english]{babel}
\begin{document}

\newtheorem{thm}{Theorem}
\newtheorem{lem}[thm]{Lemma}
\newtheorem{claim}[thm]{Claim}
\newtheorem{cor}[thm]{Corollary}
\newtheorem{prop}[thm]{Proposition}
\newtheorem{definition}[thm]{Definition}
\newtheorem{rem}[thm]{Remark}
\newtheorem{question}[thm]{Open Question}
\newtheorem{conj}[thm]{Conjecture}
\newtheorem{prob}{Problem}

\newtheorem{lemma}[thm]{Lemma}

\newcommand{\GL}{\operatorname{GL}}
\newcommand{\SL}{\operatorname{SL}}
\newcommand{\lcm}{\operatorname{lcm}}
\newcommand{\ord}{\operatorname{ord}}
\newcommand{\Op}{\operatorname{Op}}
\newcommand{\Tr}{\operatorname{Tr}}
\newcommand{\Nm}{\operatorname{Nm}}

\numberwithin{equation}{section}
\numberwithin{thm}{section}
\numberwithin{table}{section}

\def\vol {{\mathrm{vol\,}}}
\def\squareforqed{\hbox{\rlap{$\sqcap$}$\sqcup$}}
\def\qed{\ifmmode\squareforqed\else{\unskip\nobreak\hfil
\penalty50\hskip1em\null\nobreak\hfil\squareforqed
\parfillskip=0pt\finalhyphendemerits=0\endgraf}\fi}

\def \balpha{\bm{\alpha}}
\def \bbeta{\bm{\beta}}
\def \bgamma{\bm{\gamma}}
\def \blambda{\bm{\lambda}}
\def \bchi{\bm{\chi}}
\def \bphi{\bm{\varphi}}
\def \bpsi{\bm{\psi}}
\def \bomega{\bm{\omega}}
\def \btheta{\bm{\vartheta}}

\newcommand{\bfxi}{{\boldsymbol{\xi}}}
\newcommand{\bfrho}{{\boldsymbol{\rho}}}

\def\Kab{\sfK_\psi(a,b)}
\def\Kuv{\sfK_\psi(u,v)}
\def\SaUV{\cS_\psi(\balpha;\cU,\cV)}
\def\SaAV{\cS_\psi(\balpha;\cA,\cV)}

\def\SUV{\cS_\psi(\cU,\cV)}
\def\SAB{\cS_\psi(\cA,\cB)}

\def\Kmnp{\sfK_p(m,n)}

\def\KKap{\cH_p(a)}
\def\KKaq{\cH_q(a)}
\def\KKmnp{\cH_p(m,n)}
\def\KKmnq{\cH_q(m,n)}

\def\Klmnp{\sfK_p(\ell, m,n)}
\def\Klmnq{\sfK_q(\ell, m,n)}

\def \SALMNq {\cS_q(\balpha;\cL,\cI,\cJ)}
\def \SALMNp {\cS_p(\balpha;\cL,\cI,\cJ)}

\def \SACXMQX {\fS(\balpha,\bzeta, \bxi; M,Q,X)}

\def\SAMJp{\cS_p(\balpha;\cM,\cJ)}
\def\SAMJq{\cS_q(\balpha;\cM,\cJ)}
\def\SAqMJq{\cS_q(\balpha_q;\cM,\cJ)}
\def\SAJq{\cS_q(\balpha;\cJ)}
\def\SAqJq{\cS_q(\balpha_q;\cJ)}
\def\SAIJp{\cS_p(\balpha;\cI,\cJ)}
\def\SAIJq{\cS_q(\balpha;\cI,\cJ)}

\def\RIJp{\cR_p(\cI,\cJ)}
\def\RIJq{\cR_q(\cI,\cJ)}

\def\TWXJp{\cT_p(\bomega;\cX,\cJ)}
\def\TWXJq{\cT_q(\bomega;\cX,\cJ)}
\def\TWpXJp{\cT_p(\bomega_p;\cX,\cJ)}
\def\TWqXJq{\cT_q(\bomega_q;\cX,\cJ)}
\def\TWJq{\cT_q(\bomega;\cJ)}
\def\TWqJq{\cT_q(\bomega_q;\cJ)}

 \def \xbar{\overline x}
  \def \ybar{\overline y}

\def\cA{{\mathcal A}}
\def\cB{{\mathcal B}}
\def\cC{{\mathcal C}}
\def\cD{{\mathcal D}}
\def\cE{{\mathcal E}}
\def\cF{{\mathcal F}}
\def\cG{{\mathcal G}}
\def\cH{{\mathcal H}}
\def\cI{{\mathcal I}}
\def\cJ{{\mathcal J}}
\def\cK{{\mathcal K}}
\def\cL{{\mathcal L}}
\def\cM{{\mathcal M}}
\def\cN{{\mathcal N}}
\def\cO{{\mathcal O}}
\def\cP{{\mathcal P}}
\def\cQ{{\mathcal Q}}
\def\cR{{\mathcal R}}
\def\cS{{\mathcal S}}
\def\cT{{\mathcal T}}
\def\cU{{\mathcal U}}
\def\cV{{\mathcal V}}
\def\cW{{\mathcal W}}
\def\cX{{\mathcal X}}
\def\cY{{\mathcal Y}}
\def\cZ{{\mathcal Z}}
\def\Ker{{\mathrm{Ker}}}

\def\NmQR{N(m;Q,R)}
\def\VmQR{\cV(m;Q,R)}

\def\Xm{\cX_m}

\def \A {{\mathbb A}}
\def \B {{\mathbb A}}
\def \C {{\mathbb C}}
\def \F {{\mathbb F}}
\def \G {{\mathbb G}}
\def \L {{\mathbb L}}
\def \K {{\mathbb K}}
\def \N {{\mathbb N}}
\def \PP {{\mathbb P}}
\def \Q {{\mathbb Q}}
\def \R {{\mathbb R}}
\def \Z {{\mathbb Z}}
\def \fS{\mathfrak S}

\def\e{{\mathbf{\,e}}}
\def\ep{{\mathbf{\,e}}_p}
\def\eq{{\mathbf{\,e}}_q}

\def\\{\cr}
\def\({\left(}
\def\){\right)}
\def\fl#1{\left\lfloor#1\right\rfloor}
\def\rf#1{\left\lceil#1\right\rceil}

\def\Tr{{\mathrm{Tr}}}
\def\Nm{{\mathrm{Nm}}}
\def\Im{{\mathrm{Im}}}

\def \oF {\overline \F}

\newcommand{\pfrac}[2]{{\left(\frac{#1}{#2}\right)}}

\def \Prob{{\mathrm {}}}
\def\e{\mathbf{e}}
\def\ep{{\mathbf{\,e}}_p}
\def\epp{{\mathbf{\,e}}_{p^2}}
\def\em{{\mathbf{\,e}}_m}

\def\Res{\mathrm{Res}}
\def\Orb{\mathrm{Orb}}

\def\vec#1{\mathbf{#1}}
\def \va{\vec{a}}
\def \vb{\vec{b}}
\def \vn{\vec{n}}
\def \vu{\vec{u}}
\def \vv{\vec{v}}
\def \vz{\vec{z}}
\def\flp#1{{\left\langle#1\right\rangle}_p}
\def\sM {\mathsf {M}}

\def\md{{\sf{m.d.}}}

\def\sfG {\mathsf {G}}
\def\sfK {\mathsf {K}}

\def\rad{\mathrm{rad}}
 \newcommand{\Fp}{\mathbb F_p}
 \newcommand{\Gal}{\operatorname{Gal}}

\def\mand{\qquad\mbox{and}\qquad}


\title[Multiplicative dependence in recurrence sequences]{Multiplicative dependence in linear recurrence sequences}

\author[A. B\' erczes] {Attila B\' erczes}
\address{Institute of Mathematics, University of Debrecen, P. O. Box 400, H-4002 Debrecen, Hungary}
\email{berczesa@science.unideb.hu}

\author[L.  Hajdu] {Lajos Hajdu}
\address{Count Istv{\'a}n Tisza Foundation, Institute of Mathematics, University of Debrecen, 
  P. O. Box 400, H-4002 Debrecen, Hungary, and 
HUN-REN-DE Equations, Functions, Curves
and their Applications Research Group}
\email{hajdul@science.unideb.hu}

\author[A. Ostafe] {Alina Ostafe}
\address{School of Mathematics and Statistics, University of New South Wales, Sydney NSW 2052, Australia}
\email{alina.ostafe@unsw.edu.au}

\author[I. E. Shparlinski] {Igor E. Shparlinski}
\address{School of Mathematics and Statistics, University of New South Wales, Sydney NSW 2052, Australia}
\email{igor.shparlinski@unsw.edu.au}

\begin{abstract}  For a wide class of integer linear recurrence sequences $\(u(n)\)_{n=1}^\infty$,
we give an upper bound on the number of $s$-tuples $\(n_1, \ldots, n_s\) \in \(\Z\cap [M+1,M+ N]\)^s$ such that the corresponding elements
$u(n_1), \ldots, u(n_s)$ in the sequence are multiplicatively dependent.
\end{abstract}

\subjclass[2020]{11B37, 11D61}

\keywords{Multiplicative dependence, linear recurrence sequences}

\maketitle
\tableofcontents

\section{Introduction}

\subsection{Motivation and set-up}
Let  $\vu = \(u(n)\)_{n=1}^\infty$ be an integer linear recurrence sequence of order $d \ge 1$, that is, a sequence of integers satisfying a relation of the form
\[
u(n+d) = c_{d-1} u(n+d-1) +\ldots +  c_{0} u(n) , \qquad n =1,2, \ldots,
\]
and not satisfying any shorter relation. In this case
\[
f(X) = X^d -  c_{d-1} X^{n+d-1} -\ldots -  c_{0} \in \Z[X]
\]
is called the characteristic polynomial of $\vu$.

Recently there have been several works~\cite{BDKL, BeRa1, BeRa2, BeZi, LuZi, GGL, G_RL0, G_RL}  investigating multiplicative relations of the form
\begin{equation}
\label{eq:MultRel}
u(n_1)^{k_1} \ldots  u(n_s)^{k_s} = 1.
\end{equation}
However, these papers consider certain special cases. The works~\cite{BeZi, G_RL, LuZi} are limited to the case of binary (that is, of order $d =2$) linear recurrence sequences and also assume that the exponents $k_1, \ldots, k_s$ are {\it fixed\/} non-zero integers, while the papers~\cite{BDKL, BeRa1, GGL, G_RL0} concern specific sequences. Under these restrictions, the mentioned papers contain several finiteness results. Finally, the recent work~\cite{BeRa2} concerns linear recurrence sequences of arbitrary order -- however, under a rather restrictive condition on the coefficients $c_i$ defining the generating relation.

Here we are interested in the case of general sequences of arbitrary order $d \ge 2$ and also we do not fix the
exponents $k_1, \ldots, k_s$. Thus, we study $s$-tuples $\(u(n_1), \ldots, u(n_s)\)$,
which are {\it multiplicatively dependent} (\md\/), where, as usual, we say that the nonzero complex numbers $\gamma_1,\ldots,\gamma_s$ are \md\ if there exist integers $k_1,\ldots,k_s$, not all zero, such that
\[
\gamma_1^{k_1}\ldots \gamma_s^{k_s}=1.
\]
However, instead of finiteness results, we give an upper bound on the
density of such $s$-tuples.

More precisely, for $M\ge 0$ and $N\ge 1$, we are interested in the following quantity
\begin{align*}
\sM_s(M,N)=\sharp\{(n_1&,\ldots,n_s)\in\(\Z\cap [M+1,M+N]\)^s:\\
&\qquad\qquad\qquad\qquad ~u(n_1),\ldots,u(n_s)\ \text{are \md}\}.
\end{align*}
To estimate $\sM_s(M,N)$ we also study
\begin{align*}
\sM_s^*(M,N)=\sharp\{(n_1&,\ldots,n_s)\in\(\Z\cap [M+1,M+N]\)^s: \\
& ~u(n_1),\ldots,u(n_s)\ \text{are \md\ of maximal rank}\},
\end{align*}
where
the maximality of the rank for \md\ of $(u(n_1),\ldots,u(n_s))$ means that no sub-tuple is \md\ In particular, this implies that if one has a \md~\eqref{eq:MultRel} of maximal rank, then $k_1\cdots k_s \ne 0$.

We can then estimate $\sM_s(M,N)$ via the inequality
\begin{equation}
\label{eq:M and M*}
\sM_s(M,N) \le  \sum_{t=1}^s \binom{s}{t} \sM_t^*(M,N) N^{s-t} .
\end{equation}

\subsection{Notation}
We recall that  the notations $U = O(V)$, $U \ll V$ and $ V\gg U$
are equivalent to $|U|\leqslant c V$ for some positive constant $c$,
which throughout this work, may depend only on the integer parameter $s$ and the sequence $\vu$.

It is convenient to denote by  $\log_{k} x$ the $k$-fold iterated logarithm,
that is,  for $x\ge 1$ we set
\[
\log_1 x =  \log x \mand  \log_k = \log_{k-1} \max\{\log x, 2\}, \quad k =2, 3, \ldots.
\]

\subsection{Main results}

We say that the sequence $\vu$ is {\it non-degenerate\/} if there are no roots of unity among   the ratios
of distinct roots of $f$. We  say that the sequence $\vu$ has a dominant root, if its characteristic polynomial $f$
has a root $\lambda$ with
\[
|\lambda| > \max\{|\mu|:~ f(\mu) = 0, \ \mu \ne \lambda\}.
\]
Furthermore, we say that  $\vu$ is {\it simple\/} if  $f$ has no multiple roots.


\begin{thm}
\label{thm:Ms-star} Let $\vu$ be a simple non-degenerate sequence of order $d \ge 2$.
For any fixed $s \ge 1$, uniformly over $M\ge 0$, we have
\[
\sM_s^*(M,N)\le N^{s\(1-1/(4d-3)\)+o(1)}.
\]
\end{thm}

Analysing the proof of Theorem~\ref{thm:Ms-star}, one can see that for $M=0$ we can drop $o(1)$ in the  bound.

\begin{rem}  Considering $s$-tuples with $n_1=n_2$ we see  that
\begin{equation}
\label{eq:low Ms}
\sM_s(M,N) \ge N^{s-1}.
\end{equation}
Therefore, it is impossible to derive a bound of the same type as in Theorem~\ref{thm:Ms-star}  for $\sM_s(M,N)$.
\end{rem}

When $M$ is (exponentially) large compared to $N$, we get the following bound, which improves Theorem~\ref{thm:Ms-star}  for $s<4d-3$.

\begin{thm}
\label{thm:Ms-star M large} Let $\vu$ be a simple non-degenerate sequence of order $d \ge 2$ with a dominant root and let
\[M\ge \exp(N \log_3N/\log_2N).
\]
Then, for any fixed $s \ge 1$, uniformly over $M$, we have
\[
\sM_s^*(M,N)\le N^{s-1+o(1)}.
\]
\end{thm}

\begin{rem}
The condition on $M$ in Theorem~\ref{thm:Ms-star M large} is chosen to achieve the strongest
possible bound. Examining its proof one can see that for $s < 4d-3$
one can also improve Theorem~\ref{thm:Ms-star}
for $M\ge \exp(N^\eta)$ with any $\eta > s/(4d-3)$ (but only for sequences with a dominant root).  
 \end{rem}

From the definition of \md\ of maximal rank, we have $\sM_1^*(M,N) =O(1)$, see~\cite[Lemma~2.1]{AmVia}. 
Hence, we see from~\eqref{eq:M and M*}
that in applying Theorem~\ref{thm:Ms-star} to bounding $\sM_s(M,N)$
the case of $s = 2$ becomes the
bottleneck. Thus,  we now investigate this case separately.

\begin{thm}
\label{thm:M2-star}
 Let $\vu$ be a simple non-degenerate sequence of order $d \ge 2$ with a dominant root.
Uniformly over $M\ge 0$, we have
\[
\sM_2^*(M,N) = N+ O(1).
\]
\end{thm}

Since, as we have mentioned,  $\sM_1^*(M,N)=O(1)$,  the bounds of
Theorems~\ref{thm:Ms-star}  and~\ref{thm:M2-star} inserted in~\eqref{eq:M and M*}
imply that if  $\vu$ is a simple non-degenerate sequence of order $d \ge 2$ with a dominant root then
\[
\sM_s(M,N)  \ll N^{s - 3/(4d-3)+o(1)},
\]
where the bottleneck comes from the bound on $\sM_3^*(M,N)$.

If $M\ge \exp(N \log_3N/\log_2N)$, then using instead Theorem~\ref{thm:Ms-star M large}, one obtains the upper bound
\[
\sM_s(M,N)  \ll N^{s - 1+o(1)},
\]
which matches the trivial lower bound~\eqref{eq:low Ms}.

\section{Preliminaries}

\subsection{Arithmetic properties of linear recurrence sequences}
In this section we collect various results about the arithmetic properties of a linear recurrence sequence that we   need for our main results. These include:
\begin{itemize}
\item a lower bound of square-free parts of elements in $\vu$,
\item a bound for the number of elements in $\vu$ that are $S$-units,
\item various  results on congruences with elements in $\vu$,  
\item a result on the finiteness of perfect powers in $\vu$.
\end{itemize}
Some of these are obtained under the condition that $\vu$ has a dominant root.

We start with a lower bound of Stewart~\cite[Theorem~1]{Ste} on the square-free part of elements in a linear recurrence.

For any integer $m$, we define $\rad(m)$ to be the largest square-free factor of $m$.

\begin{lemma}
\label{lem:sqrfree}
Let $\vu$ be a simple non-degenerate sequence of order $d \ge 2$ with a dominant root. Then there exist constants
 $C_1$ and $C_2$, which are effectively computable only in terms of $\vu$, such that if $n\ge C_2$,  
 then
\[
\rad(u(n))>n^{C_1(\log_2 n)/\log_3 n}.
\]
\end{lemma}

We also need the following upper bound from~\cite[Theorem~1 and Corollary]{Shp2}
 on the number of terms of $\vu$ composed
out of primes from a given set.  We note that the condition of the exponential growth of the terms of $\vu$,
assumed in~\cite{Shp2},  is now known to hold for non-degenerate recurrence sequences,
see~\cite{Eve,vdPSch}.
Hence we have the following result.

\begin{lem}
\label{lem:S-unit_LRS}
Let $\vu$ be a  non-degenerate sequence of order $d \ge 2$ and let $S$ be an arbitrary set of $r$ primes. Then, for $M\ge 0$, the number  $A(S;M,N)$ of terms $u(M+1),\ldots,u(M+N)$,
composed exclusively of primes from $S$, satisfies
\[
A(S;M,N) \ll \begin{cases}
rNM^{-1}\log(N+M)& \text{for } M \ge 1,\\
r(\log N)^2& \text{for } M =0 .
\end{cases}
\]
\end{lem}

We now present two results regarding solutions to certain congruences with elements in a linear recurrence sequence.
We start with a result, which follows from~\cite[Lemma~2 and Lemma~3]{Shp2}.

\begin{lemma}
\label{lem:cong}
Let $\vu$ be a non-degenerate sequence of order $d \ge 2$ and let $m\ge 1$ be an integer.
Then we have
\[
\sharp\{n \in \Z\cap [M+1,M+N]:~u(n)\equiv 0\pmod{m}\}\ll N/\log m+1.
\]
\end{lemma}

The second  bound that we need holds modulo primes and follows from~\cite[Lemma~6]{BFKS}. In \cite{BFKS} it is formulated only for the interval $[1,N]$,
however the result is uniform with respect to the sequence
$\vu$ and hence it holds uniformly with respect to $M$, too.

Let  $\overline{\F}_p$ be the algebraic closure of the finite field $\F_p$
of $p$ elements.

\begin{lem}
\label{lem::Cong_LRS}
Let $\vu$ be a simple  sequence of order $d \ge 2$. Let $M\ge 0$ and $N\ge 1$. Let $p$ be a prime, and set
\[
\varrho_p=\min\limits_{1\leq i< j\leq d} r_{ij},
\]
where $r_{ij}$ is the multiplicative order of $\lambda_i/\lambda_j$ in $\overline{\F}_p$, with the $\lambda_i$ being the roots of the characteristic polynomial of $\vu$. Then we have
\[
\sharp\{n \in \Z\cap [M+1,M+N]:~u(n)\equiv 0\pmod{p}\}\ll N(N^{-1}+\varrho_p^{-1})^{1/(d-1)}.
\]
\end{lem}

The following result is certainly well-known and is based on classical ideas of Hooley~\cite{Hool}, however for completeness we present a short  proof.

\begin{lem}
\label{lem:MultOrd}
For $R\ge 2$ we consider the set
\[
\cW(R)=\{p\ \text{prime}:~\varrho_p\leq R\}.
\]
Then  $\sharp \cW(R) \ll R^2/\log R$.
\end{lem}

\begin{proof}
Write $\lambda_1,\dots,\lambda_q$
for the distinct roots of the characteristic polynomial of $\vu$.

  For $R \ge 2$, let
 \[
 Q(R) = \prod_{\rho \leq R} \prod_{1\leq i<j\leq q} \Nm_{K/\Q}(\lambda_i^\rho-\lambda_j^\rho),
 \]
 where $\Nm_{K/\Q}$ is the norm from the splitting field $K$ of $f$ to ${\mathbb  Q}$.
Note that $Q(R)\neq 0$ because $\lambda_i/\lambda_j$ is not a root of unity and since $\lambda_i$ and $\lambda_j$
are algebraic integers we also have $ Q(R)  \in \Z$.

Clearly, for any prime $p$ with $\varrho_p\leq R$ we have $p \mid Q(R)$, hence $\sharp \cW(R)\le \omega(Q(R))$,
where $\omega(k)$ is the number
of prime divisors of an integer $k \ge 1$.
As clearly $\omega(k)! \le k$,
by the Stirling formula we get
\[
\sharp \cW(R) \ll \frac{\log Q(R)}{\log \log Q(R)} .
\]
Since obviously $\log Q(R) \ll R^2$, the result follows.
\end{proof}

Finally, we need a result on the finiteness of perfect powers in
linear recurrence sequences with a dominant root.
The most general and convenient form for us, which is built on several previous
results in this direction,   is given by
of Bugeaud and Kaneko~\cite[Theorem~1.1]{BuKa}.

\begin{lem}
\label{lem:LRS-Pow}
Let $\vu$ be a simple non-degenerate sequence of order $d \ge 2$ with a dominant root.
Then the equation $u(n) = m^k$ has only finitely many solutions in integer
$k\ge 2$, $m \ne 0$, $n \ge 1$.
\end{lem}

\subsection{Vertex covers}
We need the following graph-theoretic result.

\begin{lem}
\label{lem:GraphCover}
Let $G$ be a graph with vertex set $\cV$, having no isolated vertex. Put $\ell=\sharp \cV$. Then there exists $ \cV_1\subseteq  \cV$ with $\sharp \cV_1\leq \ell/2$ such that for any $v_2\in  \cV_2= \cV\setminus  \cV_1$ there exists a vertex $v_1\in  \cV_1$ which is a neighbour  of $v_2$.
\end{lem}

\begin{proof} The statement must be well-known, but we give a simple proof. If $\widetilde G$ is a graph (without isolated vertices) obtained from $G$ by omitting some edges, and the statement is valid for $\widetilde G$, then the statement is obviously valid for $G$. Let $\widetilde G$ be a forest graph  (that is, a graph without cycles) obtained from $G$ by omitting some edges, such that the number of connected
components of $G$ and $\widetilde G$ are the same. Then $\widetilde G$ is a bipartite graph, so the statement is clearly valid for it. Hence the result follows.
\end{proof}

\section{Proofs}

\subsection{Proof of Theorem~\ref{thm:Ms-star}}
Suppose that for some $n_1,\ldots,n_s \in [M+1,M+N]$  the terms $u(n_1),\ldots,u(n_s)$ are \md \ of maximal rank, that is, we have~\eqref{eq:MultRel}
with some nonzero integers $k_1,\ldots,k_s$.

Choose a positive real number $R\ge 2$ to be specified later, and let $\cW(R)$ be
as in  Lemma~\ref{lem:MultOrd}.

Write $t$ for the number of indices $i =1, \ldots, s$   for which $u(n_i)$ has a prime divisor $p_i\notin \cW(R)$,
and let $r = s-t$ for the number of indices $i$ with $u(n_i)$ having all prime divisors in $\cW(R)$.
Without loss of generality, we may assume that the corresponding integers are  $n_1,\ldots,n_t$, and $n_{t+1},\ldots,n_s$, respectively.

By Lemmas~\ref{lem:S-unit_LRS} and~\ref{lem:MultOrd}, for $M\ge 1$, the number $K_1$ of such $r$-tuples $\(n_{t+1},\ldots,n_s\) \in [M+1,M+N]^r$ satisfies
\begin{equation}
\label{eq:M1}
K_1\ll \(\frac{R^2N\log (N+M)}{M\log R}\)^r.
\end{equation}
If $M=0$, then we have the bound
\begin{equation}
\label{eq:M1 M=0}
K_1\ll \(\frac{R^2(\log N)^2}{\log R}\)^r.
\end{equation}

We assume that such an $r$-tuple  $\(n_{t+1},\ldots,n_s\)$ is fixed.

Consider the $t$-tuples $\(n_1,\ldots,n_t\)\in [M+1,M+N]^t$. Recall that for any $1\leq i\leq t$, there is a prime $p_i\notin \cW(R)$ such that $p_i\mid u(n_i)$.

Define the graph $\cG$ whose vertices are $u(n_1),\ldots,u(n_t)$, and connect the vertices $u(n_i)$ and $u(n_j)$ precisely when $\gcd(u(n_i),u(n_j))$ has a prime divisor outside $\cW(R)$. Observe that  as $u(n_1),\ldots,u(n_s)$ are \md\ of
maximal rank, $\cG$ has no isolated vertex. Thus, by Lemma~\ref{lem:GraphCover}, there exists a
subset $\cI$ of $\{1,\ldots, t\}$ with
\begin{equation}
\label{eq:m t/2}
m=\sharp \cI\leq \fl{t/2}
\end{equation}
such that for any $j$ with
\[
j\in \{n_1,\ldots,n_t\}\setminus \cI
\]
the vertex $u(n_j)$ is connected with some $u(n_i)$ in $\cG$, for some $i\in \cI$.

Without loss of generality we may assume that $\cI=\{1,\ldots, m\}$. Trivially, the number $K_2$ of such $m$-tuples $(n_1,\ldots,n_m)\in[M+1,M+N]^m$ satisfies
\begin{equation}
\label{eq:M2}
K_2\ll N^m.
\end{equation}

We now fix  such an  $m$-tuple. For $\ell = t-m$, we now count  the number $K_3$ of
the remaining $\ell $-tuples $(n_{m+1},\ldots,n_t) \in [M+1,M+N]^\ell$.
Since each $u(n_j)$ with $m+1\leq j\leq t$ has a common prime factor $p\notin \cW(R)$ with $u(n_i)$ for some $1\leq i\leq m$, by Lemma~\ref{lem::Cong_LRS} we obtain that $n_j$ comes from a set $\cN$ of cardinality
\[
\sharp \cN \ll N\left(N^{-1}+\varrho_p^{-1}\right)^{1/(d-1)}\leq N\left(N^{-1}+R^{-1}\right)^{1/(d-1)}.
\]
Thus we obtain
\begin{equation}
\label{eq:M3}
K_3 \le \(\sharp \cN\)^\ell \ll \left(N\left(N^{-1}+R^{-1}\right)^{1/(d-1)}\right)^{t-m}.
\end{equation}

We consider now two cases based on $M\le N\log N$ or $M> N\log N$.

If $M\le N\log N$, then
\[
\sM_s^*(M,N)\le \sM_s^*(0,2N\log N),
\]
therefore we reduce to counting $s$-tuples in the interval $[0,2N\log N]^s$.

Putting together the bounds~\eqref{eq:M1 M=0},~\eqref{eq:M2} and~\eqref{eq:M3} (with $N$ replaced by $2N\log N$),
for some non-negative integer $t \le s$ and $r = s-t$, we obtain
\begin{equation}
\label{eq:MsN-penult}
\begin{split}
\sM_s^*(M,N)  &\le K_1K_2K_3\\
&\ll \(R^2(\log N+\log\log N)^2/\log R\)^r\(N \log N\)^m \\
& \qquad\qquad \qquad \left(N\log N\left(N^{-1}+R^{-1}\right)^{1/(d-1)}\right)^{t-m}\\
& \le  N^{t+o(1)} R^{2r} \left(\left(N^{-1}+R^{-1}\right)^{1/(d-1)}\right)^{t/2},
\end{split}
\end{equation}
where the last inequality comes from~\eqref{eq:m t/2}.

Letting $R=N^\eta$ with some $0<\eta<1/2$,   we obtain
\[
\sM_s^*(M,N)\ll N^{t + 2\eta  r  - \eta t/(2(d-1)) + o(1)}
=  N^{2\eta  s + (1-2 \eta) t   - \eta t/(2(d-1)) + o(1)} .
\]
Writing $t=\ccg{z}s$ (and noting that $0\leq z\leq 1$), the exponent of the last term above (omitting the expression $o(1)$) is given by
\[
f_\eta(z) = \frac{s}{2(d-1)}((2d-4d\eta+3\eta-2)z+4\eta(d-1)).
\]
So taking
\[
\eta = \frac{2(d-1)}{4d-3},
\]
(to make $f_\eta(z)$ a constant), we obtain
\[
\sM_s^*(M,N)\ll  N^{2\eta  s +  o(1)}  = N^{s - s/(4d-3) + o(1)},
\]
which concludes this case.

If $M>N\log N$, then the bound~\eqref{eq:M1} becomes
\[
K_1\ll (R^2/\log R)^r.
\]
Putting this together with~\eqref{eq:M2} and~\eqref{eq:M3}, we obtain~\eqref{eq:MsN-penult} without the $(\log N)^2$ factor, that is,
\begin{align*}
\sM_s^*(M,N)& \le K_1K_2K_3\\
& \ll (R^2/\log R)^rN^m \left(N\left(N^{-1}+R^{-1}\right)^{1/(d-1)}\right)^{t-m}\\
& \ll N^{t} R^{2r} \left(\left(N^{-1}+R^{-1}\right)^{1/(d-1)}\right)^{t/2}.
\end{align*}
Using the same discussion and choice of $\eta$ as above, we conclude the proof.
\qed

\begin{rem}  Clearly in~\eqref{eq:MsN-penult}  we can replace $t/2$ with $\rf{t/2}$
but this does not change the optimal choice of $\eta$ and thus the final bound.
\end{rem}

\subsection{Proof of Theorem~\ref{thm:Ms-star M large}}
Let $(n_1,\ldots,n_s)\in[M+1,M+N]^s$ such that $u(n_1),\ldots,u(n_s)$ is \md\ of maximal rank, which implies that there exist integers $k_i\ne 0$, $i=1,\ldots,s$, such that~\eqref{eq:MultRel} holds. We can rewrite this relation as
\begin{equation}
\label{eq:md pos exp}
\prod_{i \in \cI}u(n_i)^{k_i} = \prod_{j\in \cJ} u(n_j)^{k_j},\quad k_i,k_j>0,
\end{equation}
where $\cI\cup \cJ=\{1,\ldots,s\}$, $\cI\neq \emptyset$, $\cJ\neq \emptyset$, $\cI\cap \cJ=\emptyset$. Let $I=\sharp\cI$ and $J=\sharp\cJ$, and thus, $I+J=s$.

Fix one of $2^s-2$ possible  choices of  the sets $\cI$ and $\cJ$ as above.
Fix $n_i$, $i\in \cI$, trivially in $O(N^I)$ ways.  Then, the square-free part $\rad(u(n_i))$ of $u(n_i)$
is fixed for each $i\in\cI$.

We may also assume that $n_i\ge C_2$, $i\in\cI$, with $C_2$ as in Lemma~\ref{lem:sqrfree}, since this condition is violated only
for $O(N^{s-1})$
choices of $(n_1, \ldots, n_s)$, which is admissible. 
By Lemma~\ref{lem:sqrfree}, for $n_i\in[M+1,M+N]$, one has
\begin{equation}
\label{eq:rad un}
\rad(u(n_i))>n_i^{c(\log_2 n_i)/\log_3 n_i}\gg M^{c(\log_2 M)/\log_3(M+N)}.
\end{equation}

For $i\in \cI$, from~\eqref{eq:md pos exp} we see
\[
\rad(u(n_i)) \mid \prod_{j\in \cJ} u(n_j).
\]
This implies that there is a  
factorisation $\rad(u(n_i))=d_1\cdots d_J$ such that for each positive integer $d_\ell$ there exists $j\in\cJ$ such that $d_\ell\mid u(n_{j})$.
Let $\ell$, $1 \le \ell \le J$, be such that $d_\ell\ge \rad(u(n_i))^{1/J}$, and
\begin{equation}
\label{eq:cong pk}
u(n_j)\equiv 0 \pmod{d_\ell}.
\end{equation}
From~\eqref{eq:rad un} we have
\begin{equation}
\label{eq:rad un J}
d_\ell>M^{c_0(\log_2 M)/\log_3(M+N)}
\end{equation}
with $c_0 = c/J \ge c/s$.

Using now Lemma~\ref{lem:cong}, the inequality~\eqref{eq:rad un J} and the fact that $J\le s$, the number of $n_j\in[M+1,M+N]$ satisfying the congruence~\eqref{eq:cong pk} is \[
O\(N/\log d_{\ell} +1\)=O\( N\frac{\log_3(M+N)}{\log M \log_2 M}+1\).
\]
Therefore, using the trivial bound $N^{J-1}$ for the number  of the remaining choices of $n_j$ with $j\in \cJ$, we obtain that the total number of $n_j\in[M+1,M+N]$, $j\in \cJ$, is
\[
O\(N^J\frac{\log_3(M+N)}{\log M \log_2 M}+N^{J-1}\).
\]
Thus we obtain that
\[
\sM_s^*(M,N)\ll N^s \frac{\log_3(M+N)}{\log M \log_2 M}+ N^{s-1}.
\]
Choosing $M\ge \exp(N \log_3N/\log_2N)$, we conclude the proof.
\qed

\subsection{Proof of Theorem~\ref{thm:M2-star}}
Clearly for  $s=2$ we have to count integers $M+1\le m, n \le M+N$,  with
\begin{equation}
\label{eq:uma=nub}
u(m)^a = u(n)^b
\end{equation}
for some positive integers $a$ and $b$, where without
loss of generality we can assume that  $\gcd(a,b) = 1$. We also notice that since the relation~\eqref{eq:uma=nub} is of maximal rank,
neither $u(m)=\pm 1$ nor $u(n)=\pm 1$ holds.

Since $\vu$ has a dominant root,  $|u(n)|$ grows monotonically with $n$,
provided that $n$ is large enough. Hence there are $N+O(1)$ solutions $(m,n)\in [M+1,M+N]^2$ with
$a=b=1$.

Now we count pairs $(m,n)$ for which~\eqref{eq:uma=nub}
holds with some $(a,b) \ne (1,1)$.

We observe that if $a>1$ then $u(n)$ is the $a$-th power
and by Lemma~\ref{lem:LRS-Pow} there are $O(1)$ such values of $n$.
For $b > 1$ the argument also applies to $m$. Hence the total contribution
from such solutions,  over all $a, b > 1$, is $O(1)$.

If $a> 1$ and $b = 1$ then again  we see that there are $O(1)$
such values of $n$.  From this we easily derive that $a= O(1)$,
and hence we obtain $O(1)$ possible values for $m$.
So the contribution of such solutions to~\eqref{eq:uma=nub}  is also O(1) only.

The case of $a=1$ and $b >1$ is completely analogous, which concludes the proof.
\qed

\section{Possible applications of our approach}

Our approach works for many other integer sequences $\(a(n)\)_{n=1}^\infty$, 
provided the following information is available:

\begin{itemize}
\item[(i)]  there are good bounds on the number of solutions to congruences $a(n) \equiv 0 \pmod q$, $1 \le n \le N$, 
in a broad range of positive integers $q$ (or even just prime $q=p$) and $N$;

\item[(ii)] there are good bounds (or known finiteness) on the number of perfect powers among $a(n)$, 
$1 \le n \le N$.

\end{itemize}

For example, using results of~\cite{ShZa}, coupled with the finiteness result of Lemma~\ref{lem:LRS-Pow}, 
one can estimate the number of multiplicatively dependent $s$-tuples from values of linear recurrence sequences 
at polynomial values of the argument $\(u\(F(n)\)\)_{n=1}^\infty$, where $F\in \Z[X]$.


\section*{Acknowledgement}

During this work,  A.B. and L.H. were  supported, in part, by the
NKFIH grants 130909 and 150284 and A.O.  and I.S.  by the Australian Research Council Grant  DP230100530.  A.O. gratefully acknowledges the hospitality and support of
the University of Debrecen, where this work was initiated,  and the Max Planck Institute for Mathematics and Institut des Hautes Études Scientifiques, where  this
work has been carried out.


\begin{thebibliography}{99}

\bibitem{AmVia}
F. Amoroso and E. Viada, `On the zeros of linear recurrence sequences',
{\it Acta Arith.\/}, {\bf  147} (2011), 387--396.

\bibitem{BFKS}  W. D. Banks,
J. B. Friedlander,  S. V. Konyagin and I. E. Shparlinski,
  `Incomplete exponential sums and Diffie--Hellman triples',
  {\it Math. Proc. Camb. Phil. Soc.\/}, {\bf 140} (2006), 193--206.

\bibitem{BDKL} H. Batte, M. Ddamulira, J. Kasozi and F. Luca,
'Multiplicative independence in the sequence of $k$-generalized Lucas numbers', {\it Indag. Math.\/}, in press, DOI: 10.1016/j.indag.2024.09.002.

\bibitem{BeRa1}
M. Behera and P. K. Ray, 'Multiplicative dependence between Padovan and Perrin numbers', {\it Bol. Soc. Mat. Mexicana\/}, {\bf 29} (2023), 52.

\bibitem{BeRa2}
M. Behera and P. K. Ray, 'On multiplicative dependent vectors with coordinates from certain higher order recurrences',
{\it Int. J. Number Theory\/}, {\b  20} (2024),   2491--2507.

\bibitem{BeZi} A. B{\'e}rczes  and V. Ziegler,
`On geometric progressions on Pell equations and Lucas sequences',
  {\it Glas. Mat.\/},  {\bf 48} (2013),  1--22.


\bibitem{BuKa} Y. Bugeaud and H. Kaneko,
`On perfect powers in linear recurrence sequences of integers',
{\it   Kyushu J. Math.\/},  {\bf  73} (2019), 221--227.

\bibitem{Eve} J.-H. Evertse, `On sums of $S$-units and linear
recurrences', {\it Compositio Math.\/}, {\bf 53} (1984), 225--244.

\bibitem{GGL} C. A. G{\'o}mez, J. C. G{\'o}mez and F. Luca, `Multiplicative dependence between $k$-Fibonacci and $k$-Lucas numbers', {\it Period. Math. Hungar.\/}, {\bf 81} (2020), 217--233.

\bibitem{G_RL0} C. A. G{\'o}mez Ruiz and F. Luca, `Multiplicative independence in $k$-generalized Fibonacci sequences', {\it Lith. Math. J.\/}, {\bf 56} (2016), 503--517.

\bibitem{G_RL} C. A. G{\'o}mez Ruiz and F. Luca, `Multiplicative Diophantine equations with factors from different Lucas sequences', {\it J. Number Theory\/},  {\bf 170} (2017), 282--301.


 \bibitem{Hool} C. Hooley, `Artin's conjecture for primitive
 roots', {\it J. Reine Angew. Math.}, {\bf 225} (1967), 209--220.

\bibitem{LuZi} F.  Luca  and V.  Ziegler,
`Multiplicative relations on binary recurrences',
{\it Acta Arith.\/}, {\bf 161}  (2013), 183--199.


\bibitem{vdPSch} A. J. van der Poorten and H. P. Schlickewei, `Additive
relations in fields', {\it J. Austral Math. Soc.\/}, {\bf 51}
(1991), 154--170.

\bibitem{Shp2} I. E.  Shparlinski,
`Some arithmetic properties of recurrence sequences',
{\it Math. Notes\/}, {\bf 47} (1990), 612--617 (translated from
{\it Matem. Zametki\/}).

\bibitem{ShZa} I. E. Shparlinski and U. Zannier,
`Arithmetic properties of quadratic exponential polynomials',
{\it New York J. Math.\/}, {\bf 25} (2019) 207--218.


\bibitem{Ste} C. L. Stewart, `On the greatest square-free factor of terms of a linear recurrence sequence',
{\it Diophantine Equations\/}, Tata Inst. Fund. Res. Stud. Math. 20, Tata Inst.
Fund. Res., Mumbai, 2008, 257--264.


\end{thebibliography}
\end{document}